\documentclass[11pt]{amsart}
\usepackage{enumerate,amssymb,amsfonts,amsmath}
\usepackage{comment}
\usepackage{hyperref}
\usepackage{xcolor}

\newtheorem{theorem}{Theorem}
\newtheorem{definition}{Definition}
\newtheorem{lemma}{Lemma}
\newtheorem{cs}{Step}

\def\diam{\mathrm{diam}}
\def\Re{\mathrm{Re}}
\def\Im{\mathrm{Im}}

\usepackage{vmargin}
\setpapersize{A4}
\setmargins{2.5cm}      
{1.5cm}                        
{16.5cm}                      
{23.42cm}                   
{10pt}                          
{1cm}                           
{0pt}                          
{2cm}

\begin{document}

\title[Algebraic reflexivity of diameter-preserving linear bijections]{Algebraic reflexivity of diameter-preserving linear bijections between $C(X)$-spaces}

\author{A. Jim{\'e}nez-Vargas}
\address{Departamento de Matem{\'a}ticas, Universidad de Almer{\'i}a, 04120, Almer{\i}a, Spain}
\email{ajimenez@ual.es}

\author{Fereshteh Sady}
\address{Department of Pure Mathematics, Faculty of Mathematical Sciences, Tarbiat Modares University, Tehran 14115-134, Iran}
\email{sady@modares.ac.ir}

\date{\today}


\begin{abstract}
We prove that if $X$ and $Y$ are first countable compact Hausdorff spaces, then the set of all diameter-preserving linear bijections from $C(X)$ to $C(Y)$ is algebraically reflexive. 
\end{abstract}
\maketitle

\section{Introduction and statement of the result}

This paper is concerned with the algebraic reflexivity of the set of all diameter-preserving linear bijections between $C(X)$-spaces. We shall denote by $C(X)$ the Banach algebra of all continuous complex-valued functions on a compact Hausdorff space $X$, with the usual supremum norm. 

Our interest focuses on the local behaviour of linear maps on $C(X)$ which preserve the diameter of the ranges of functions in $C(X)$. Let us recall that for compact Hausdorff spaces $X$ and $Y$,  a map $T$ from $C(X)$ into $C(Y)$  is said to be diameter-preserving if $\diam(T(f))=\diam(f)$ for all $f\in C(X)$, where $\diam(f)$ denotes the diameter of $f(X)$. 

Gy\H{o}ry and Moln\'{a}r \cite{GyoMol-98} introduced this kind of maps and stated the general form of diameter-preserving linear bijections of $C(X)$, when $X$ is a first countable compact Hausdorff space. Cabello S\'anchez \cite{Cab-99} and, independently, Gonz\'{a}lez and Uspenskij \cite{GonUsp-99} extended this description by removing the hypothesis of first countability. Namely, they proved the following 

\begin{theorem}\label{IsoLipWeaver}\cite{Cab-99, GonUsp-99, GyoMol-98}. 
Let $X$ and $Y$ be compact Hausdorff spaces. A linear bijection $T\colon C(X)\to C(Y)$ is diameter-preserving if and only if there exist a homeomorphism $\phi\colon Y\to X$, a linear functional $\mu\colon C(X)\to\mathbb{C}$ and a number $\lambda\in\mathbb{C}$ with $|\lambda|=1$ and $\lambda\neq-\mu(1_X)$ such that 
$$
T(f)(y)=\lambda f(\phi(y))+\mu(f)
$$
for every $y\in Y$ and $f\in C(X)$.
\end{theorem}

The statement of Theorem \ref{IsoLipWeaver} also holds for the algebra of continuous real-valued functions on $X$.

A problem addressed by different authors is the Banach--Stone type representation of diameter-preserving maps between function spaces. See, for example, the papers by Aizpuru and Rambla \cite{AizRam-07}, Barnes and Roy \cite{BarRoy-02}, Font and Hosseini \cite{FonHos-19}, Gy\H{o}ry \cite{Gyo-99}, Jamshidi and Sady \cite{JamSad-16}, and Rao and Roy \cite{RaoRoy-01}.

On the other hand, a linear map $T$ of $C(X)$ into itself is called a local isometry (respectively, local automorphism) if for every $f\in C(X)$, there exists a surjective linear isometry (respectively, automorphism) $T_f$ of $C(X)$, depending on $f$, such that $T(f)=T_f(f)$.

It is said that the set of all surjective linear isometries (respectively, automorphisms) of $C(X)$ is algebraically reflexive if every local isometry (respectively, local automorphism) of $C(X)$ is a surjective linear isometry (respectively, automorphism) of $C(X)$. 

The algebraic reflexivity of both sets of surjective linear isometries and automorphisms of $C(X)$ was stated by Moln\'ar and Zalar in \cite[Theorem 2.2]{MolZal-99} whenever $X$ is a first countable compact Hausdorff space. Furthermore, Cabello and Moln\'ar \cite{CabMol-02} gave an example where that reflexivity fails even if $X$ lacks first countability at only one point. Afterwards, the algebraic reflexivity of some function spaces has been studied by Botelho and Jamison \cite{BotJam-11}, Cabello S\'{a}nchez and Moln\'{a}r \cite{CabMol-02}, Dutta and Rao \cite{DutRao-08}, Jarosz and Rao \cite{JaRao-03}, and Oi \cite{Oi19}, among others.

Motivated by the precedent considerations, we introduce the following concept.

\begin{definition}\label{def:local diameter-preserving map}
Let $X$ and $Y$ be compact Hausdorff spaces. A linear map $T\colon C(X)\to C(Y)$ is local diameter-preserving if for every $f\in C(X)$, there exists a diameter-preserving linear bijection $T_f\colon C(X)\to C(Y)$, depending on $f$, such that $T(f)=T_f(f)$. 

We say that the set of all diameter-preserving linear bijections from $C(X)$ to $C(Y)$ is algebraically reflexive if every local diameter-preserving linear map from $C(X)$ to $C(Y)$ is a diameter-preserving bijection.
\end{definition}

Our main result is the following. 

\begin{theorem}\label{main}
Let $X$ and $Y$ be first countable compact Hausdorff spaces. Then the set of all diameter-preserving linear bijections from $C(X)$ to $C(Y)$ is algebraically reflexive.
\end{theorem}

Our proof consists in showing that every local diameter-preserving linear map $T$ from $C(X)$ to $C(Y)$ can be expressed in the form 
$$
T(f)(y)=\lambda f(\phi(y))+\mu(f)\qquad \left(y\in Y,\; f\in C(X)\right),
$$
with $\lambda$, $\phi$ and $\mu$ being as in the statement of Theorem \ref{IsoLipWeaver}. Using this representation, it is proven easily that $T$ is surjective. 

Our approach is in the line of the proof of the known Holszty\'{n}ski theorem \cite{Hol-66} which provides a Banach--Stone type representation for non-surjective linear isometries of $C(X)$ with the supremun norm. However, the adaptation of the Holszty\'{n}ski's method to the setting of diameter-preserving linear maps is far from being immediate due to the representation of the diameter-preserving linear bijections from $C(X)$ to $C(Y)$ as sum of a weighted composition operator from $C(X)$ to $C(Y)$ and a linear functional on $C(X)$.

We shall apply the known Gleason--Kahane--Zelazko theorem \cite{Gle-67, KahZel-68, Zel-68} to prove our main result. A similar strategy was used in the study of local isometries between complex-valued Lipschitz algebras \cite{JimMorVil-10} or uniform algebras \cite{CabMol-02}. Recently, Li, Peralta, L. Wang and Y.-S. Wang \cite{LiPerWanWan-19} established a spherical variant of the Gleason--Kahane--Zelazko theorem to analyse weak-local isometries on uniform algebras and Lipschitz algebras.

\section{Proof of Theorem \ref{main}}

Before proving our result, we fix some notation and recall the existence of certain peaking functions. Given a set $X$ with cardinal number $|X|\geq 2$, we denote 
\begin{align*}
\widetilde{X}&=\left\{(x_1,x_2)\in X\times X\colon x_1\neq x_2\right\},\\
X_2&=\left\{\{x_1,x_2\}\colon (x_1,x_2)\in\widetilde{X}\right\}.
\end{align*}
As usual, $\mathbb{T}$ stands for the set of all unimodular complex numbers. We also denote 
$$
\mathbb{T}^+=\left\{e^{it}\colon t\in [0,\pi[\right\}.
$$

An application of Urysohn's lemma 
shows that if $X$ is a first countable compact Hausdorff space and $(x_1,x_2)\in\widetilde{X}$, then there exists a continuous function $h_{(x_1,x_2)}\colon X\to [0,1]$ with $h_{(x_1,x_2)}^{-1}(\{1\})=\{x_1\}$ and $h_{(x_1,x_2)}^{-1}(\{0\})=\{x_2\}$, 
and hence 
$$
h_{(x_1,x_2)}(x_1)-h_{(x_1,x_2)}(x_2)=1=\diam(h_{(x_1,x_2)})
$$
and 
$$
\left\{(x,y)\in\widetilde{X}\colon h_{(x_1,x_2)}(x)-h_{(x_1,x_2)}(y)=1\right\}=\left\{(x_1,x_2)\right\}.
$$
Therefore, given a first countable compact Hausdorff space $X$ and any $\{x_1,x_2\}\in X_2$, we may consider the nonempty sets:
\begin{align*}
\mathcal{F}_{\{x_1,x_2\}}&=\left\{f\in C(X)\colon \left|f(x_1)-f(x_2)\right|=1=\diam(f)\right\},\\
\mathcal{F'}_{\{x_1,x_2\}}&=\left\{f\in\mathcal{F}_{\{x_1,x_2\}}\colon \left\{\{x,y\}\in X_2\colon \left|f(x)-f(y)\right|=1\}=\{\{x_1,x_2\}\right\}\right\}.
\end{align*}

We should note that,  since the range of a local diameter-preserving linear map is a subspace without any additional (separating) property, the  standard reasoning  does not work here in  some steps of the proof. Indeed,  we need the next two lemmas. The first one provides some functions in  $\mathcal{F}'_{\{x_1,x_2\}}$ satisfying an additional condition and the second one shows, in particular,  that $C(X)$ is  the linear span of $\bigcup_{\{x_1,x_2\}\in X_2} \mathcal{F}_{\{x_1,x_2\}}$.  

\begin{lemma}\label{0}
Let $X$ be a first countable compact Hausdorff space and let $x_1,x_2,x_3,x_4$ be pairwise distinct points in $X$. Then there exists a function $f\in\mathcal{F}'_{\{x_1,x_2\}}$ for which $f(x_3)=f(x_4)$.
\end{lemma}

\begin{proof}
We construct the function $f$ in several stages:

1. Choose $f_0\in \mathcal{F}'_{\{x_1,x_2\}}$ with values in $[0,1]$ such that $f_0^{-1}(\{1\})=\{x_2\}$ and $f_0^{-1}(\{0\})=\{x_1\}$. 
If $f_0(x_3)=f_0(x_4)$, then $f_0$ is the desired function. 
So we assume that $f_0(x_3)\neq f_0(x_4)$. Put $a=f_0(x_3)$ and $b=f_0(x_4)$, and assume without loss of generality that $a<b$. Clearly, $0<a<b<1$. 

2. Let $U$ and $V$ be neighbourhoods of $x_3$ and $x_4$, respectively, with $\overline{U}\cap\overline{V}=\emptyset$ and $x_2\notin\overline{U}\cup \overline{V}$. Choose $g_0\in C(X)$ satisfying $g_0\le 0$, $g_0(x_3)=\ln(b)$, $g_0(x_4)=\ln(a)$ and $g_0=0$ on $X\backslash (U\cup V)$. 
For such a function it suffices to take $h_0,h_1\in C(X)$ with values in $[0,1]$ such that $h_0(x_3)=1$ and ${\rm supp}(h_0)\subseteq U$ and similarly 
$h_1(x_4)=1$ and ${\rm supp}(h_1)\subseteq V$. Then $g_0=\ln(b)h_0+\ln(a)h_1$ has the desired properties.

3. Put $g=e^{g_0}$. Since $g_0\le 0$, we have $0<g\le 1$. Clearly, $g(x_3)=b$, $g(x_4)=a$ and $g(x_2)=1$. 

4. Take $f=f_0 g$. Then $f(x_2)=1$, $f(x_1)=0$ and 
$$
f(x_3)=f_0(x_3)g(x_3)=ab=f_0(x_4)g(x_4)=f(x_4).
$$
Moreover, for any $x\in X$, we have $f(x)=1$ only if $f_0(x)=1$, i.e., $x=x_2$. Similarly, $f(x)=0$ if and only if $f_0(x)=0$, i.e., $x=x_1$. Hence $0<f(x)<1$ for all $x\notin \{x_1,x_2\}$. This implies that $f\in \mathcal{F}'_{\{x_1,x_2\}}$ and this completes the proof. 
\end{proof}
\begin{lemma} \cite[Lemma 2.1(i)]{JamSad-16}\label{00}
Let $X$ be a compact Hausdorff space and $x_1,x_2\in X$ be distinct. If $f\in C(X)$ such that  $0\le f\le 1$ and $f(x_1)=f(x_2)$, then there exists a function $g\in C(X)$ such that both $g$ and $h:=\frac{1}{2}f+g$ satisfy  $g(x_1)-g(x_2)=1={\rm diam}(g)$ and   $h(x_1)-h(x_2)=1={\rm diam}(h)$. In particular, we have  $g,h\in F_{\{x_1,x_2\}}$. 
\end{lemma}
Let $T$ be a local diameter-preserving linear map from $C(X)$ to $C(Y)$. We have divided the proof of Theorem \ref{main} in several steps. 

\begin{cs}\label{Step 1}
$T$ is diameter-preserving.
\end{cs}

\begin{proof}
Let $f\in C(X)$. By hypothesis, there is a diameter-preserving linear bijection $T_{f}$ from $C(X)$ to $C(Y)$ such that $T(f)=T_{f}(f)$. Hence $\diam(T(f))=\diam(T_{f}(f))=\diam(f)$. 
\end{proof}

\begin{cs}\label{Step 2}
For every $f\in C(X)$, there exist a homeomorphism $\phi_f\colon Y\to X$, a linear functional $\mu_f$ on $C(X)$ and a number $\lambda_f\in\mathbb{T}$ with $\lambda_f\neq -\mu_f(1_X)$ such that 
$$
T(f)(y)=\lambda_f f(\phi_f(y))+\mu_f(f)
$$
for all $y\in Y$.
\end{cs}

\begin{proof}
It follows immediately from Definition \ref{def:local diameter-preserving map} and Theorem \ref{IsoLipWeaver}. 
\end{proof}

Step \ref{Step 2} will be frequently applied without any explicit mention along the paper. By Step \ref{Step 2}, there exists a homeomorphism from $Y$ onto $X$. Hence $|Y|=|X|$. Since Theorem \ref{main} is easy to verify when $|Y|=1$, we shall suppose $|Y|\geq 2$ from now on.

\begin{cs}\label{Step 3}
For every $(x_1,x_2)\in\widetilde{X}$, the set
$$ 
\mathcal{B}_{(x_1,x_2)}=\bigcap_{f\in\mathcal{F}_{\{x_1,x_2\}}}\mathcal{B}_{(x_1,x_2),f}
$$
is nonempty, where
$$
\mathcal{B}_{(x_1,x_2),f}=\left\{((y_1,y_2),\lambda)\in\widetilde{Y}\times \mathbb{T}\colon T(f)(y_1)-T(f)(y_2)=\lambda\left(f(x_1)-f(x_2)\right)\right\}\qquad \left(f\in\mathcal{F}_{\{x_1,x_2\}}\right).
$$
\end{cs}

\begin{proof}
Let $(x_1,x_2)\in\widetilde{X}$. We shall first prove that for each $f\in\mathcal{F}_{\{x_1,x_2\}}$, the set $\mathcal{B}_{(x_1,x_2),f}$ is a nonempty closed subset of $\widetilde{Y}\times \mathbb{T}$. Fix $f\in\mathcal{F}_{\{x_1,x_2\}}$ and take $y_1,y_2\in Y$ such that $\phi_f(y_1)=x_1$ and $\phi_f(y_2)=x_2$. Clearly, $y_1\neq y_2$. We have   
$$
T(f)(y_1)-T(f)(y_2)
=\lambda_f\left(f(\phi_f(y_1))-f(\phi_f(y_2))\right)
=\lambda_f\left(f(x_1)-f(x_2)\right),
$$
and thus $((y_1,y_2),\lambda_f)\in\mathcal{B}_{(x_1,x_2),f}$. Therefore $\mathcal{B}_{(x_1,x_2),f}$ is nonempty, and to prove that it is closed in $\widetilde{Y}\times\mathbb{T}$, assume that $\{((y_i,z_i),\lambda_i)\}_{i\in I}$ is a net in $\mathcal{B}_{(x_1,x_2),f}$ converging to $((y_1,y_2),\lambda)$ in $\widetilde{Y}\times\mathbb{T}$ equipped with the product topology. We have 
$$
T(f)(y_i)-T(f)(z_i)=\lambda_i\left(f(x_1)-f(x_2)\right)
$$
for all $i\in I$. Since $T(f)\in C(Y)$, we infer that 
$$
T(f)(y_1)-T(f)(y_2)=\lambda\left(f(x_1)-f(x_2)\right),
$$
and thus $((y_1,y_2),\lambda)\in\mathcal{B}_{(x_1,x_2),f}$. A similar reasoning shows that $\mathcal{B}_{(x_1,x_2),f}$ is a nonempty closed subset of $Y^2\times \mathbb{T}$.

We shall next prove that the family $\left\{\mathcal{B}_{(x_1,x_2),f}\colon f\in\mathcal{F}_{\{x_1,x_2\}}\right\}$ has the finite intersection property. Let $n\in\mathbb{N}$ and $f_1,\ldots,f_n\in\mathcal{F}_{\{x_1,x_2\}}$. Define the function $g\colon X\to\mathbb{C}$ by 
$$
g(x)=\frac{1}{n}\sum_{i=1}^n(\overline{f_i(x_1)-f_i(x_2)})f_i(x).
$$
It is clear that $g\in C(X)$ with $g(x_1)-g(x_2)=1$. By Step \ref{Step 2}, consider $\lambda_g\in \mathbb{T}$ and take $y_1,y_2\in Y$ such that $\phi_g(y_1)=x_1$ and $\phi_g(y_2)=x_2$. Clearly, $y_1\neq y_2$. We have
$$
T(g)(y_1)-T(g)(y_2)=\lambda_g\left(g(\phi_g(y_1))-g(\phi_g(y_2))\right)=\lambda_g\left(g(x_1)-g(x_2)\right)=\lambda_g.
$$
Using the linearity of $T$, we can write 
$$
\lambda_g=T(g)(y_1)-T(g)(y_2)=\frac{1}{n}\sum_{i=1}^n(\overline{f_i(x_1)-f_i(x_2)})\left(T(f_i)(y_1)-T(f_i)(y_2)\right).
$$
By Step \ref{Step 1}, note that 
\begin{align*}
\left|(\overline{f_i(x_1)-f_i(x_2)})(T(f_i)(y_1)-T(f_i)(y_2))\right|&=\left|T(f_i)(y_1)-T(f_i)(y_2)\right|\\
                                                                    &\leq\diam(T(f_i))=\diam(f_i)=1
\end{align*}
for every $i\in\{1,\ldots,n\}$. By the strict convexity of $\mathbb{C}$, it follows that 
$$
T(f_i)(y_1)-T(f_i)(y_2)=\lambda_g\left(f_i(x_1)-f_i(x_2)\right)
$$
for all $i\in\{1,\ldots,n\}$, and thus $((y_1,y_2),\lambda_g)\in\cap_{i=1}^n\mathcal{B}_{(x_1,x_2),f_i}$, as desired.

Hence $\left\{\mathcal{B}_{(x_1,x_2),f}\colon f\in \mathcal{F}_{\{x_1,x_2\}}\right\}$ is a family of closed subsets of the compact space $Y^2\times \mathbb{T}$ with the finite intersection property. Therefore there exists $((y_1,y_2),\lambda)\in Y^2\times \mathbb{T}$ such that 
$$
T(f)(y_1)-T(f)(y_2)=\lambda\left(f(x_1)-f(x_2)\right)
$$
for any $f\in\mathcal{F}_{\{x_1,x_2\}}$. This implies $y_1\neq y_2$ and thus $((y_1,y_2),\lambda)\in\mathcal{B}_{(x_1,x_2)}$.
\end{proof}

\begin{cs}\label{Step 3'5} 
For every $(x_1,x_2)\in\widetilde{X}$, there exist $(y_1,y_2)\in\widetilde{Y}$ and $\lambda\in\mathbb{T}^+$ such that 
$$
\mathcal{B}_{(x_1,x_2)}=\left\{((y_1,y_2),\lambda),((y_2,y_1),-\lambda)\right\}.
$$
\end{cs}

\begin{proof}
Let $(x_1,x_2)\in\widetilde{X}$. By Step \ref{Step 3}, the set $\mathcal{B}_{(x_1,x_2)}$ is nonempty. Hence we can take an element $((y_1,y_2),\lambda)\in\mathcal{B}_{(x_1,x_2)}$. Note that $((y_2,y_1),-\lambda)\in\mathcal{B}_{(x_1,x_2)}$. Let $((z_1,z_2),\beta)\in\mathcal{B}_{(x_1,x_2)}$ be arbitrary. We have
\begin{align*}
T(f)(y_1)-T(f)(y_2)&=\lambda\left(f(x_1)-f(x_2)\right),\\
T(f)(z_1)-T(f)(z_2)&=\beta\left(f(x_1)-f(x_2)\right),
\end{align*}
for all $f\in\mathcal{F}_{\{x_1,x_2\}}$. Fix any $f\in\mathcal{F'}_{\{x_1,x_2\}}$. Using Step \ref{Step 2}, we deduce 
\begin{align*}
\lambda_f\left(f(\phi_f(y_1))-f(\phi_f(y_2))\right)&=\lambda\left(f(x_1)-f(x_2)\right),\\
\lambda_f\left(f(\phi_f(z_1))-f(\phi_f(z_2))\right)&=\beta\left(f(x_1)-f(x_2)\right).
\end{align*}
Since $f\in\mathcal{F'}_{\{x_1,x_2\}}$ and  
$$
\left|f(\phi_f(y_1))-f(\phi_f(y_2))\right|=\left|f(\phi_f(z_1))-f(\phi_f(z_2))\right|=1,
$$
we derive from above that  
$$
\left\{(\phi_f(y_1),\phi_f(y_2)),(\phi_f(z_1),\phi_f(z_2))\right\}\subseteq\left\{(x_1,x_2),(x_2,x_1)\right\}.
$$
We have four possibilities:
\begin{enumerate}
\item $x_1=\phi_f(y_1),\; x_2=\phi_f(y_2),\; x_1=\phi_f(z_1),\; x_2=\phi_f(z_2)$.
\item $x_1=\phi_f(y_1),\; x_2=\phi_f(y_2),\; x_1=\phi_f(z_2),\; x_2=\phi_f(z_1)$.
\item $x_1=\phi_f(y_2),\; x_2=\phi_f(y_1),\; x_1=\phi_f(z_2),\; x_2=\phi_f(z_1)$.
\item $x_1=\phi_f(y_2),\; x_2=\phi_f(y_1),\; x_1=\phi_f(z_1),\; x_2=\phi_f(z_2)$.
\end{enumerate}
Using the injectivity of $\phi_f$, we deduce that 
$$
((z_1,z_2),\beta)\in\left\{((y_1,y_2),\lambda),((y_2,y_1),-\lambda)\right\}.
$$
Therefore 
$$
\mathcal{B}_{(x_1,x_2)}=\left\{((y_1,y_2),\lambda),((y_2,y_1),-\lambda)\right\}.
$$
Finally, notice that either $\lambda\in \mathbb{T}^+$ or $-\lambda\in \mathbb{T}^+$.
\end{proof}

\begin{cs}\label{Step 4} 
For every $(x_1,x_2)\in \widetilde{X}$, the set 
$$
\mathcal{A}_{(x_1,x_2)}=\left\{(y_1,y_2)\in\widetilde{Y}\,|\,  \exists \lambda\in \mathbb{T}^+ \colon ((y_1,y_2),\lambda)\in\mathcal{B}_{(x_1,x_2)}\right\}
$$
is a singleton. Let $\Gamma\colon\widetilde{X}\to\widetilde{Y}$ be the map given by
$$
\left\{\Gamma(x_1,x_2)\right\}=\mathcal{A}_{(x_1,x_2)}.
$$
Furthermore, $(y_2,y_1)=\Gamma(x_2,x_1)$ if $(y_1,y_2)=\Gamma(x_1,x_2)$. 
\end{cs}

\begin{proof}
Given $(x_1,x_2)\in \widetilde{X}$, the set $\mathcal{A}_{(x_1,x_2)}$ is a singleton by Step \ref{Step 3'5}, say $\mathcal{A}_{(x_1,x_2)}=\left\{(y_1,y_2)\right\}$. Hence $\Gamma(x_1,x_2)=(y_1,y_2)\in \widetilde{Y}$. Let $(x_1,x_2),(x_3,x_4)\in \widetilde{X}$ be such that $(x_1,x_2)=(x_3,x_4)$. Let $\Gamma(x_1,x_2)=(y_1,y_2)\in \widetilde{Y}$. Hence $(y_1,y_2)\in\mathcal{A}_{(x_1,x_2)}$ and therefore there exists $\lambda\in \mathbb{T}^+$ such that $((y_1,y_2),\lambda)\in\mathcal{B}_{(x_1,x_2)}$. 
It follows that $((y_1,y_2),\lambda)\in\mathcal{B}_{(x_3,x_4)}$,  
hence $(y_1,y_2)\in\mathcal{A}_{(x_3,x_4)}$ and so $(y_1,y_2)=\Gamma(x_3,x_4)$. Consequently, $\Gamma(x_1,x_2)=\Gamma(x_3,x_4)$. This justifies that the map $\Gamma\colon \widetilde{X}\to\widetilde{Y}$ is well-defined.

For the last statement, let $(y_1,y_2)=\Gamma(x_1,x_2)$. Then $(y_1,y_2)\in\mathcal{A}_{(x_1,x_2)}$ and therefore there exists $\lambda\in \mathbb{T}^+$ such that $((y_1,y_2),\lambda)\in\mathcal{B}_{(x_1,x_2)}$. It follows that $((y_2,y_1),\lambda)\in\mathcal{B}_{(x_2,x_1)}$, hence $(y_2,y_1)\in\mathcal{A}_{(x_2,x_1)}$ and thus $(y_2,y_1)=\Gamma(x_2,x_1)$, as required. 
\end{proof}

\begin{cs}\label{Step 4'5} 
If $(x_1,x_2)\in \widetilde{X}$ and $(y_1,y_2)=\Gamma(x_1,x_2)$, then  
$$
T(f)(y_1)=T(f)(y_2)
$$
for all $f\in C(X)$ such that $f(x_1)=f(x_2)$. 
\end{cs}

\begin{proof}
Let $(x_1,x_2)\in \widetilde{X}$ and $(y_1,y_2)=\Gamma(x_1,x_2)$. By Step \ref{Step 4}, there is a $\beta(x_1,x_2)\in \mathbb{T}^+$ such that 
$$
T(h)(y_1)-T(h)(y_2)=\beta(x_1,x_2)\left(h(x_1)-h(x_2)\right)
$$
for all $h\in\mathcal{F}_{\{x_1,x_2\}}$. Let $f$ be in $C(X)$ with $f(x_1)=f(x_2)$. Assume first that $0\leq f\leq 1$. By Lemma \ref{00}, we can take a function $g\in\mathcal{F}_{\{x_1,x_2\}}$ such that $(1/2)f+g\in\mathcal{F}_{\{x_1,x_2\}}$. Therefore 
$$
T\left(\frac{1}{2}f+g\right)(y_1)-T\left(\frac{1}{2}f+g\right)(y_2)=\beta(x_1,x_2)\left(\left(\frac{1}{2}f+g\right)(x_1)-\left(\frac{1}{2}f+g\right)(x_2)\right).
$$
Using the linearity of $T$ and the equality  
$$
T(g)(y_1)-T(g)(y_2)=\beta(x_1,x_2)\left(g(x_1)-g(x_2)\right),
$$
we get  
$$
T(f)(y_1)=T(f)(y_2).
$$ 
If $f$ is arbitrary, consider the decomposition 
$$
f=(\Re f)^+-(\Re f)^-+i[(\Im f)^+-(\Im f)^-],
$$
apply the previous case to each one of the four functions of the decomposition $f/(1+||f||_\infty)$, and the same conclusion is achieved by using the linearity of $T$.
\end{proof}

\begin{cs}\label{Step 5} 
For every $(x_1,x_2)\in\widetilde{X}$, there exists a number $\lambda(x_1,x_2)\in \mathbb{T}^+$ such that
$$
T(f)(y_1)-T(f)(y_2)=\lambda(x_1,x_2)\left(f(x_1)-f(x_2)\right)                                             
$$
for all $f\in C(X)$, where $(y_1,y_2)=\Gamma(x_1,x_2)$. Furthermore, $\lambda(x_1,x_2)=\lambda(x_2,x_1)$.
\end{cs}

\begin{proof}
Let $(x_1,x_2)\in\widetilde{X}$ and $(y_1,y_2)=\Gamma(x_1,x_2)$. Let $\lambda(x_1,x_2)$ be the number given by  
$$
\lambda(x_1,x_2)=T(f)(y_1)-T(f)(y_2),                                             
$$
where $f$ is any function in $C(X)$ which satisfies $f(x_1)-f(x_2)=1$. The number $\lambda(x_1,x_2)$ does not depend on such a function $f$ by Step \ref{Step 4'5}, and it is well-defined. Using the homogeneity of $T$, we may deduce easily that  
$$
T(f)(y_1)-T(f)(y_2)=\lambda(x_1,x_2)\left(f(x_1)-f(x_2)\right)                                             
$$
for all $f\in C(X)$. 

Since $(y_1,y_2)=\Gamma(x_1,x_2)$, Step \ref{Step 4} gives a $\lambda\in \mathbb{T}^+$ such that 
$$
T(f)(y_1)-T(f)(y_2)=\lambda\left(f(x_1)-f(x_2)\right)                                             
$$
for all $f\in\mathcal{F}_{\{x_1,x_2\}}$. In particular, taking $f=h_{(x_1,x_2)}$ yields
$$
\lambda(x_1,x_2)=T(h_{(x_1,x_2)})(y_1)-T(h_{(x_1,x_2)})(y_2)=\lambda ,
$$
and so $\lambda(x_1,x_2)\in \mathbb{T}^+$.   

Similarly, since $(y_2,y_1)=\Gamma(x_2,x_1)$ by Step \ref{Step 4}, we have 
$$
T(f)(y_2)-T(f)(y_1)=\lambda(x_2,x_1)\left(f(x_2)-f(x_1)\right)                                             
$$
for all $f\in C(X)$. Combining the equations obtained, we infer that 
\begin{align*}
\lambda(x_1,x_2)\left(f(x_1)-f(x_2)\right)&=T(f)(y_1)-T(f)(y_2)\\
                                          &=-\left(T(f)(y_2)-T(f)(y_1)\right)\\
                                          &=-\lambda(x_2,x_1)\left(f(x_2)-f(x_1)\right)\\
                                          &=\lambda(x_2,x_1)\left(f(x_1)-f(x_2)\right)
\end{align*}
for all $f\in C(X)$, and taking $f=h_{(x_1,x_2)}$ yields $\lambda(x_1,x_2)=\lambda(x_2,x_1)$. 
\end{proof}

\begin{cs}\label{Step 6}
The map $\Gamma$ is a bijection from $\widetilde{X}$ to $\cup_{(x_1,x_2)\in \widetilde{X}}\mathcal{A}_{(x_1,x_2)}$.
\end{cs}

\begin{proof}
Let $(y_1,y_2)\in\cup_{(x_1,x_2)\in \widetilde{X}}\mathcal{A}_{(x_1,x_2)}$. Then $(y_1,y_2)\in\mathcal{A}_{(x_1,x_2)}$ for some $(x_1,x_2)\in \widetilde{X}$. By Step \ref{Step 4}, $\mathcal{A}_{(x_1,x_2)}=\left\{(y_1,y_2)\right\}$, and thus $\Gamma(x_1,x_2)=(y_1,y_2)$ by the definition of $\Gamma$. This proves the surjectivity of $\Gamma$. 

To prove its injectivity, let $(x_1,x_2),(x_3,x_4)\in \widetilde{X}$ be such that
$$
(y_1,y_2)=\Gamma(x_1,x_2)=\Gamma(x_3,x_4),
$$
where $(y_1,y_2)\in \cup_{(x_1,x_2)\in \widetilde{X}}\mathcal{A}_{(x_1,x_2)}$. By Step \ref{Step 5}, we have 
$$
\lambda(x_1,x_2)\left(f(x_1)-f(x_2)\right)=T(f)(y_1)-T(f)(y_2)=\lambda(x_3,x_4)\left(f(x_3)-f(x_4)\right)                                
$$
for all $f\in C(X)$, with $\lambda(x_1,x_2),\lambda(x_3,x_4)\in \mathbb{T}^+$. Substituting any function $f\in\mathcal{F}'_{\{x_1,x_2\}}$, we deduce that $\{x_3,x_4\}=\{x_1,x_2\}$. This implies that either $(x_1,x_2)=(x_4,x_3)$ or $(x_1,x_2)=(x_3,x_4)$. In the former case, we would have 
$$
\lambda(x_1,x_2)\left(f(x_1)-f(x_2)\right)=\lambda(x_3,x_4)\left(f(x_2)-f(x_1)\right)=-\lambda(x_3,x_4)\left(f(x_1)-f(x_2)\right)                               
$$
for all $f\in C(X)$. In particular, for $f=h_{(x_1,x_2)}$ we would obtain $\lambda(x_1,x_2)=-\lambda(x_3,x_4)$, which is impossible. Therefore $(x_1,x_2)=(x_3,x_4)$. 
\end{proof}

\begin{cs}\label{Step 7}
Let $(x_1,x_2),(x_3,x_4)\in\widetilde{X}$, $(y_1,y_2)=\Gamma(x_1,x_2)$ and $(y_3,y_4)=\Gamma(x_3,x_4)$. Then   
$$
\left|\{x_1,x_2\}\cap\{x_3,x_4\}\right|=\left|\{y_1,y_2\}\cap\{y_3,y_4\}\right|.
$$
With others words, if $\Lambda_X\colon\widetilde{X}\to X_2$ and $\Lambda_Y\colon\widetilde{Y}\to Y_2$ are the maps defined by $\Lambda_X(x_1,x_2)=\{x_1,x_2\}$ and $\Lambda_Y(y_1,y_2)=\{y_1,y_2\}$, respectively, we have 
$$
\left|\Lambda_X(x_1,x_2)\cap\Lambda_X(x_3,x_4)\right|=\left|\Lambda_Y(\Gamma(x_1,x_2))\cap\Lambda_Y(\Gamma(x_3,x_4))\right|
$$
for all $(x_1,x_2),(x_3,x_4)\in\widetilde{X}$.
\end{cs}

\begin{proof}
Firstly, assume $\left|\{x_1,x_2\}\cap\{x_3,x_4\}\right|=2$.
Then $(x_1,x_2)\in\{(x_3,x_4),(x_4,x_3)\}$, hence $(y_1,y_2)\in\{\Gamma(x_3,x_4),\Gamma(x_4,x_3)\}=\{(y_3,y_4),(y_4,y_3)\}$ and thus 
$$
\left|\{y_1,y_2\}\cap\{y_3,y_4\}\right|=2.
$$
Secondly, if $\left|\{x_1,x_2\}\cap\{x_3,x_4\}\right|=1$, then $\left|\{y_1,y_2\}\cap\{y_3,y_4\}\right|\leq 1$ by the injectivity of $\Gamma$, and therefore
$$
\left|\{y_1,y_2\}\cap\{y_3,y_4\}\right|=1.
$$
Indeed, we can assume without loss of generality that $x_1=x_3$ and $x_2\neq x_4$, and assume on the contrary that $\left|\{y_1,y_2\}\cap\{y_3,y_4\}\right|=0$. By Step \ref{Step 5}, we have the equations
\begin{align*}
T(f)(y_1)-T(f)(y_2)&=\lambda(x_1,x_2)(f(x_1)-f(x_2)),\\
T(f)(y_3)-T(f)(y_4)&=\lambda(x_1,x_4)(f(x_1)-f(x_4)),
\end{align*} 
for all $f\in C(X)$. Since the finite sets in a first countable compact space $X$ are $G_\delta$-sets, it is possible to take a continuous function $f\colon X\to [0,1]$ such that $f^{-1}(\{1\})=\{x_1\}$ and $f^{-1}(\{0\})=\{x_2,x_4\}$, and, consequently, 
$$
\left\{\{x,y\}\in X_2\colon |f(x)-f(y)|=1\}=\{\{x_1,x_2\},\{x_1,x_4\}\right\}.
$$
From the equations, it follows that  
\begin{align*}
\lambda_f\left(f(\phi_f(y_1))-f(\phi_f(y_2))\right)&=\lambda(x_1,x_2)(f(x_1)-f(x_2)),\\
\lambda_f\left(f(\phi_f(y_3))-f(\phi_f(y_4))\right)&=\lambda(x_1,x_4)(f(x_1)-f(x_4)),
\end{align*} 
which imply that $\{\phi_f(y_1),\phi_f(y_2)\},\{\phi_f(y_3),\phi_f(y_4)\}\in\left\{\{x_1,x_2\},\{x_1,x_4\}\right\}$. In any case, we deduce that $\phi_f(y_i)=x_1=\phi_f(y_j)$ for some $i\in\{1,2\}$ and $j\in\{3,4\}$ with $i\neq j$. Since $\phi_f$ is injective, we get $y_i=y_j$, a contradiction

Finally, assume $\left|\{x_1,x_2\}\cap\{x_3,x_4\}\right|=0$, then $\left|\{y_1,y_2\}\cap\{y_3,y_4\}\right|\leq 1$ by the injectivity of $\Gamma$, and we shall prove that 
$$
\left|\{y_1,y_2\}\cap\{y_3,y_4\}\right|=0.
$$
Assume on the contrary that $y_1=y_3$ and $y_2\neq y_4$ (in the other cases it is proved in a similar form). Then we have the two equations:
\begin{align*}
T(f)(y_1)-T(f)(y_2)&=\lambda(x_1,x_2)(f(x_1)-f(x_2)),\\
T(f)(y_1)-T(f)(y_4)&=\lambda(x_3,x_4)(f(x_3)-f(x_4)),
\end{align*} 
for all $f\in C(X)$. It follows that 
$$
T(f)(y_4)-T(f)(y_2)=\lambda(x_1,x_2)(f(x_1)-f(x_2))-\lambda(x_3,x_4)(f(x_3)-f(x_4))
$$
for all $f\in C(X)$. Since  $\{x_1,x_2\}\cap\{x_3,x_4\}=\emptyset$, Lemma \ref{0} provides a function $f\in\mathcal{F}'_{\{x_1,x_2\}}$ satisfying $f(x_3)=f(x_4)$. Hence we have 
$$
\lambda_f\left(f(\phi_f(y_4))-f(\phi_f(y_2))\right)
=\lambda(x_1,x_2)(f(x_1)-f(x_2)),
$$
which implies that $\{\phi_f(y_4),\phi_f(y_2)\}=\{x_1,x_2\}$. Using the first one of the above-mentioned equations, we also obtain $\{\phi_f(y_1),\phi_f(y_2)\}=\{x_1,x_2\}$. These equalities imply that $\phi_f(y_4)=\phi_f(y_1)$ and, since $\phi_f$ is injective, we get $y_4=y_1$, hence $y_4=y_3$, a contradiction.
\end{proof}

\begin{cs}\label{Step 8}
Assume $|X|\geq 3$. For each $x\in X$ and any $(x_1,x_2)\in \widetilde{X}$ with $x_1\neq x\neq x_2$, there exists a unique point, depending only on $x$ and denoted by $\varphi(x)$, in the intersection 
$$
\Lambda_Y(\Gamma(x,x_1))\cap\Lambda_Y(\Gamma(x,x_2)).
$$
The map $\varphi\colon X\to Y$ so defined is injective and $\{\varphi(x_1),\varphi(x_2)\}=\Lambda_Y(\Gamma(x_1,x_2))$ for every $(x_1,x_2)\in \widetilde{X}$. 
\end{cs}

\begin{proof}
Let $x\in X$ and let $x_1,x_2\in X$ be with $x_1\neq x_2$ and $x_1\neq x\neq x_2$. Let $y$ be the unique point of the set $\Lambda_Y(\Gamma(x,x_1))\cap\Lambda_Y(\Gamma(x,x_2))$ (see Step \ref{Step 7}). 

We claim that $y\in\Lambda_Y(\Gamma(x,x_3))$ for every $x_3\in X$ with $x_3\neq x$, what shows that $y$ does not depend on $x_1$ and $x_2$ and thus it depends only on $x$. Indeed, if $\left|X\right|=3$, this is obvious. Assume $\left|X\right|\geq 4$. Pick $x_3\in X\setminus\{x,x_1,x_2\}$ and suppose on the contrary that $y\notin\Lambda_Y(\Gamma(x,x_3))$. We can write $\Lambda_Y(\Gamma(x,x_1))=\{y,y_1\}$ and $\Lambda_Y(\Gamma(x,x_2))=\{y,y_2\}$ for some $y_1,y_2\in Y$ with $y_1\neq y\neq y_2$. In the light of Step \ref{Step 7}, we obtain $y_1\neq y_2$. 
Since the cardinal of both sets $\Lambda_Y(\Gamma(x,x_3))\cap\Lambda_Y(\Gamma(x,x_1))$ and $\Lambda_Y(\Gamma(x,x_3))\cap\Lambda_Y(\Gamma(x,x_2))$ is one, we deduce that $\Lambda_Y(\Gamma(x,x_3))=\{y_1,y_2\}$. This implies that $\Gamma(x,x_3)=(y_1,y_2)$ or $\Gamma(x,x_3)=(y_2,y_1)$. We shall only prove the first case and the other is similarly proven. Since $\lambda(x,x_3),\lambda(x,x_1),\lambda(x,x_2)\in\mathbb{T}^+$, an easy argument shows that  
\begin{align*}
\lambda(x,x_3)(f(x)-f(x_3))&=T(f)(y_1)-T(f)(y_2)\\
                           &=(T(f)(y_1)-T(f)(y))+(T(f)(y)-T(f)(y_2))\\
                           &=\lambda(x,x_1)(f(x)-f(x_1))+\lambda(x,x_2)(f(x)-f(x_2))
\end{align*}
for all $f\in C(X)$. Taking suitable functions $f\in C(X)$, we can deduce that 
$$
\lambda(x,x_3)=\lambda(x,x_1)=\lambda(x,x_2),
$$ 
and so $f(x)=f(x_1)+f(x_2)-f(x_3)$ for all $f\in C(X)$, which is impossible. This proves our claim.
 
We shall next prove the injectivity of $\varphi$. Suppose first $|X|=3$, say $X=\{x_1,x_2,x_3\}$. If $\varphi(x_1)=\varphi(x_2)=y_1$, then $y_1\in\Lambda_Y(\Gamma(x_1,x_2))\cap\Lambda_Y(\Gamma(x_1,x_3))\cap\Lambda_Y(\Gamma(x_2,x_3))$. As the cardinality of each one of the three sets in this intersection is $2$, there are $y_2,y_3,y_4\in Y\setminus\{y_1\}$ such that $\Lambda_Y(\Gamma(x_1,x_2))=\{y_1,y_2\}$, $\Lambda_Y(\Gamma(x_1,x_3))=\{y_1,y_3\}$ and $\Lambda_Y(\Gamma(x_2,x_3))=\{y_1,y_4\}$. Applying Step \ref{Step 7} yields $y_2\neq y_3\neq y_4\neq y_2$, and thus $|Y|\geq 4$ which contradicts that $|X|=|Y|$. 

Assume now $|X|\geq 4$. Let $x_1,x_2\in X$ be with $x_1\neq x_2$ and suppose $\varphi(x_1)=\varphi(x_2)=y_2$. Take $\{z_1,z_2\}\in X_2$ such that $\{z_1,z_2\}\cap\{x_1,x_2\}=\emptyset$. We have $y_2\in\Lambda_Y(\Gamma(x_1,z_1))\cap\Lambda_Y(\Gamma(x_2,z_2))$; but since $|\Lambda_X(x_1,z_1)\cap\Lambda_X(x_2,z_2)|=0$, we have $|\Lambda_Y(\Gamma(x_1,z_1))\cap\Lambda_Y(\Gamma(x_2,z_2))|=0$ by Step \ref{Step 7}, a contradiction. This completes the proof that $\varphi$ is injective. 

For the second assertion, note that if $(x_1,x_2)\in \widetilde{X}$, then $\varphi(x_1)$ and $\varphi(x_2)$ are distinct and belong to $\Lambda_Y(\Gamma(x_1,x_2))$ (see Step \ref{Step 4}). Hence $\{\varphi(x_1),\varphi(x_2)\}=\Lambda_Y(\Gamma(x_1,x_2))$.
\end{proof}

\begin{cs}\label{Step 9}
There exist a nonempty subset $Y_0\subseteq Y$ and a bijection $\phi_0\colon Y_0\to X$ such that $\{y_1,y_2\}=\Lambda_Y(\Gamma(\phi_0(y_1),\phi_0(y_2)))$ for all $y_1,y_2\in Y_0$ with $y_1\neq y_2$.
\end{cs}

\begin{proof}
Assume first $|X|=2$. Then $|Y|=2$ by Step \ref{Step 2}. Hence $X=\{x_1,x_2\}$ and $Y=\{y_1,y_2\}$ for certain $(x_1,x_2)\in\widetilde{X}$ and $(y_1,y_2)\in\widetilde{Y}$. Clearly, $\widetilde{X}=\left\{(x_1,x_2),(x_2,x_1)\right\}$ and $\widetilde{Y}=\left\{(y_1,y_2),(y_2,y_1)\right\}$. Since $\Gamma$ is a map from $\widetilde{X}$ to $\widetilde{Y}$, we have $\Lambda_Y(\Gamma(x_1,x_2))=\{y_1,y_2\}$. Take $Y_0=Y$ and the bijection $\phi_0\colon Y_0\to X$ defined by $\phi_0(y_1)=x_1$ and $\phi_0(y_2)=x_2$, and the proof is finished if $|X|=2$.

Assume now $|X|\geq 3$. Let $\varphi\colon X\to Y$ be the injective map defined in Step \ref{Step 8}. Then $Y_0=\varphi(X)$ and $\phi_0=\varphi^{-1}\colon Y_0\to X$ satisfy the required conditions. 
\end{proof}

\begin{cs}\label{Step 10}
There exists a number $\lambda\in \mathbb{T}$ such that
$$
T(f)(y_1)-T(f)(y_2)=\lambda\left(f(\phi_0(y_1))-f(\phi_0(y_2))\right)                                             
$$
for all $y_1,y_2\in Y_0$ and $f\in C(X)$.
\end{cs}

\begin{proof}
Let $Y_0\subseteq Y$ and $\phi_0\colon Y_0\to X$ be the set and the bijection given in Step \ref{Step 9}. Let $y_1,y_2\in Y_0$ with $y_1\neq y_2$. By Step \ref{Step 9}, $\{y_1,y_2\}=\Lambda_Y(\Gamma(\phi_0(y_1),\phi_0(y_2)))$. Hence either $\Gamma(\phi_0(y_1),\phi_0(y_2))=(y_1,y_2)$ or $\Gamma(\phi_0(y_1),\phi_0(y_2))=(y_2,y_1)$. By Step \ref{Step 5},  we have 
$$
T(f)(y_1)-T(f)(y_2)=\pm\lambda(\phi_0(y_1),\phi_0(y_2))\left(f(\phi_0(y_1))-f(\phi_0(y_2))\right)
$$
for all $f\in C(X)$, where $\lambda(\phi_0(y_1),\phi_0(y_2))\in \mathbb{T}^+$. Put $\beta(\phi_0(y_1),\phi_0(y_2))\in\{\pm\lambda(\phi_0(y_1),\phi_0(y_2))\}$. 

We now claim that $\beta(\phi_0(y_1),\phi_0(y_2))$ does not depend on its variables $y_1,y_2$. It is clear when $|Y_0|=2$ because $\beta(\phi_0(y_1),\phi_0(y_2))=\beta(\phi_0(y_2),\phi_0(y_1))$ by Step \ref{Step 5}. Otherwise, let $y_3\in Y_0$ be with $y_3\notin\{y_1,y_2\}$. We have the equation
\begin{align*}
\beta&(\phi_0(y_1),\phi_0(y_2))\left(f(\phi_0(y_1))-f(\phi_0(y_2))\right)\\
                            &=T(f)(y_1)-T(f)(y_2)\\
								            &=\left(T(f)(y_1)-T(f)(y_3)\right)+\left(T(f)(y_3)-T(f)(y_2)\right) \\
                            &=\beta(\phi_0(y_1),\phi_0(y_3))\left(f(\phi_0(y_1))-f(\phi_0(y_3))\right)+\beta(\phi_0(y_3),\phi_0(y_2))\left(f(\phi_0(y_3))-f(\phi_0(y_2))\right)
\end{align*}
for all $f\in C(X)$. For each $i\in\{1,2\}$, consider the set 
$$
F_i=\left\{\phi_0(y_1),\phi_0(y_2),\phi_0(y_3)\right\}\setminus\left\{\phi_0(y_i)\right\}
$$
and take a function $f_i\in C(X)$ satisfying $f_i(x)=0$ for all $x\in F_i$ and $f_i(\phi_0(y_i))=1$. Taking $f=f_i$ for $i=1,2$ in the equation above, it follows that
$$
\beta(\phi_0(y_1),\phi_0(y_3))=\beta(\phi_0(y_1),\phi_0(y_2))=\beta(\phi_0(y_3),\phi_0(y_2)),
$$
as claimed. Indeed, by the arbitrariness of $y_1$, $y_2$ and $y_3$, the first equality in the preceding equation means that the function $\beta(\cdot,\cdot)$ does not depend on the second variable, while the second equality says us that the same occurs with the first one. Hence there exist a constant $\lambda\in \mathbb{T}$ such that $\beta(\phi_0(y_1),\phi_0(y_2))=\lambda$ for all $y_1,y_2\in Y_0$ with $y_1\neq y_2$. 

Now we get 
$$
T(f)(y_1)-T(f)(y_2)
=\beta(\phi_0(y_1),\phi_0(y_2))\left(f(\phi_0(y_1))-f(\phi_0(y_2))\right)
=\lambda\left(f(\phi_0(y_1))-f(\phi_0(y_2))\right)                                             
$$
for all $f\in C(X)$ and $y_1,y_2\in Y_0$. 
\end{proof}

\begin{cs}\label{Step 11}
There exists a linear functional $\mu\colon C(X)\to\mathbb{C}$ such that
$$
T(f)(y)=\lambda f(\phi_0(y))+\mu(f)                                             
$$
for every $y\in Y_0$ and $f\in C(X)$.
\end{cs}

\begin{proof}
Define a functional $\mu\colon C(X)\to\mathbb{C}$ by 
$$
\mu(f)=T(f)(y)-\lambda f(\phi_0(y))                                             
$$
for all $f\in C(X)$, where $y$ is an arbitrary point in $Y_0$. By Step \ref{Step 10}, $\mu$ is well-defined. Since $T$ is linear, so is $\mu$. 
\end{proof}

\begin{cs}\label{Step 12}
$\lambda\neq -\mu(1_X)$.
\end{cs}

\begin{proof}
By Step \ref{Step 2}, we have 
$$
T(1_X)(y)=\lambda_{1_X}+\mu_{1_X}(1_X)
$$
for all $y\in Y$, with $\lambda_{1_X}\in \mathbb{T}$ and $\lambda_{1_X}\neq -\mu_{1_X}(1_X)$. On the other hand, by Step \ref{Step 11}, we have 
$$
T(1_X)(y)=\lambda+\mu(1_X)                                             
$$
for all $y\in Y_0$. Hence $\lambda+\mu(1_X)=\lambda_{1_X}+\mu_{1_X}(1_X)\neq 0$.
\end{proof}

\begin{cs}\label{Step 13}
$\phi_0\colon Y_0\to X$ is a homeomorphism.
\end{cs}

\begin{proof}
We first prove that $\phi_0$ is continuous. Let $y\in Y_0$ and let $\{y_i\}_i$ be a net in $Y_0$ which converges to $y$. Since $X$ is compact, taking a subnet if necessary, we can suppose that $\{\phi_0(y_i)\}_i$ converges to some $x_0\in X$. We claim that $x_0=\phi_0(y)$. Otherwise, we could find an open neighborhood $U$ of $x_0$ in $X$ such that $\phi_0(y)\in X\setminus U$. Take a function $f\in C(X)$ which satisfies $f(x_0)=1$ and $f(x)=0$ for all $x\in X\setminus U$. There exists $i_0\in I$ such that $|f(\phi_0(y_i))-f(x_0)|=|f(\phi_0(y_i))-1|<1/3$ for all $i\geq i_0$ and, by Step \ref{Step 11}, it follows that $|T(f)(y_i)-T(f)(y)|=|f(\phi_0(y_i))|>2/3$ for all $i\geq i_0$, which contradicts the continuity of $T(f)$. This proves our claim and so $\phi_0$ is continuous.

We next show that $Y_0$ is closed in $Y$. Since $Y_0=Y$ in the case $|X|=2$, we suppose $|X|\geq 3$. Let $\{y_i\}_i$ be a net in $Y_0$ convergent to some point $y\in Y$. By the compactness of $X$, taking a subnet if necessary, we can suppose that $\{\phi_0(y_i)\}_i$ converges to some $x_1\in X$. Given $x_2\in X\setminus\{x_1\}$, there exists $y_2\in Y_0$ such that $\phi_0(y_2)=x_2$. By Step \ref{Step 11}, we have 
$$
T(f)(y_i)-T(f)(y_2)=\lambda\left(f(\phi_0(y_i))-f(\phi_0(y_2))\right)=\lambda\left(f(\phi_0(y_i))-f(x_2)\right)                                             
$$
for each $f\in C(X)$ and all $i\in I$. Since $f$ and $T(f)$ are continuous, taking limits in $i$ above, it follows that  
$$
T(f)(y)-T(f)(y_2)=\lambda(f(x_1)-f(x_2))
$$
for all $f\in C(X)$. Note that $y\neq y_2$ since $C(X)$ separates the points of $X$. In particular, we get
$$
T(f)(y)-T(f)(y_2)=\lambda \left(f(x_1)-f(x_2)\right)
$$
for all $f\in\mathcal{F}_{\{x_1,x_2\}}$. Hence $((y,y_2),\lambda)\in\mathcal{B}_{(x_1,x_2)}$. By Steps \ref{Step 3'5} and \ref{Step 4}, we have either $(y,y_2)\in\mathcal{A}_{(x_1,x_2)}$ or $(y_2,y)\in\mathcal{A}_{(x_1,x_2)}$. Hence either $\Gamma(x_1,x_2)=(y,y_2)$ or $\Gamma(x_1,x_2)=(y_2,y)$ by Step \ref{Step 4}. Therefore $\{y,y_2\}=\Lambda_Y(\Gamma(x_1,x_2))=\{\varphi(x_1),\varphi(x_2)\}$ by Step \ref{Step 8}, and so $y\in\varphi(X)=Y_0$. 

Finally, since $\phi_0\colon Y_0\to X$ is bijective and continuous, $Y_0$ is compact and $X$ is Hausdorff, then $\phi_0$ is a homeomorphism. 
\end{proof}

We have $|Y|=|X|=|Y_0|$ by Steps \ref{Step 2} and \ref{Step 9}. If $Y$ is finite, then $Y_0=Y$ since $Y_0\subseteq Y$, and we would obtain Steps \ref{Step 14} and \ref{Step 15} taking $\phi=\phi_0$. Suppose that $Y$ is not finite henceforth. 

\begin{cs}\label{Step 14}
There exists a continuous map $\phi\colon Y\to X$ such that 
$$
T(f)(y)=\lambda f(\phi(y))+\mu(f)
$$
for all $f\in C(X)$ and $y\in Y$. 
\end{cs} 

\begin{proof}
For each $y\in Y$, define the linear functional $S_y\colon C(X)\to\mathbb{C}$ by 
$$
S_y(f)=T(f)(y)-\mu(f) \qquad (f\in C(X)),
$$
with $\mu\colon C(X)\to\mathbb{C}$ being as in Step \ref{Step 11}. Note that $T(1_X)(y_0)=\lambda+\mu(1_X)$ for each $y_0\in Y_0$ by Step \ref{Step 11}. Since $T(1_X)$ is a constant function by Step \ref{Step 1}, it follows that $T(1_X)=(\lambda+\mu(1_X))1_Y$. Hence $S_y(1_X)=\lambda$.  

We shall now prove that $\lambda^{-1}S_y$ is multiplicative. By the Gleason--Kahane--Zelazko theorem, it suffices to show that for each non-vanishing function $f\in C(X)$, we have $S_y(f)\neq 0$. For this, let $f\in C(X)$ be with $f(x)\neq 0$ for all $x\in X$ and assume on the contrary that $T(f)(y)=\mu(f)$. Being $\phi_0\colon Y_0 \to X$ a bijective map, there exists $y_0\in Y_0$ such that 
$$
\phi_0(y_0)=\phi_f(y).
$$
In the same way we can find a sequence $\{y_i\}_{i=0 }^\infty$ in $Y_0$ satisfying 
$$
\phi_0(y_{i+1})=\phi_f(y_{i}) \qquad (i\in \mathbb{N}\cup \{0\}).
$$
Since $Y$ is a compact (first countable) space, passing through a subsequence we may assume that $\{y_i\}_i\to z_0$ for some $z_0\in Y_0$. Hence, tending $i\to \infty$ in the above equality, we get 
$$
\phi_0(z_0)=\phi_f(z_0).
$$
For each $i\in \mathbb{N}\cup \{0\}$, since $z_0,y_i\in Y_0$, Step \ref{Step 10} yields 
\begin{align*}
\lambda\left(f(\phi_0(z_0))-f(\phi_0(y_i))\right)&=T(f)(z_0)-T(f)(y_i)\\
&=\lambda_f\left(f(\phi_f(z_0))-f(\phi_f(y_i))\right) \\
&=\lambda_f\left(f(\phi_0(z_0))-f(\phi_0(y_{i+1}))\right).
\end{align*}
Hence, for each $i\in \mathbb{N}\cup \{0\}$ we have 
$$
f(\phi_0(z_0))-f(\phi_0(y_i))=\lambda^{-1}\lambda_f\left(f(\phi_0(z_0))-f(\phi_0(y_{i+1}))\right).
$$
For each $i\in \mathbb{N}\cup \{0\}$, it follows by induction on $n$ that  
$$
f(\phi_0(z_0))-f(\phi_0(y_i))=(\lambda^{-1} \lambda_f)^n \left(f(\phi_0(z_0))-f(\phi_0(y_{i+n})\right)
$$
for all $n\in \mathbb{N}$. Now, tending $n\to \infty$, we get 
$$
f(\phi_0(z_0))=f(\phi_0(y_i)) \qquad (i\in \mathbb{N}\cup \{0\}). 
$$
Therefore, we have 
$$
T(f)(z_0)=\lambda f(\phi_0(z_0))+\mu(f)=\lambda f(\phi_0(y_0))+\mu(f).
$$
On the other hand, since $f(\phi_f(y))=f(\phi_0(y_0))=f(\phi_0(z_0))$ and $\phi_f(z_0)=\phi_0(z_0)$, we also get   
\begin{align*}
T(f)(y)&=\lambda_f f(\phi_f(y))+\mu_f(f)\\
       &=\lambda_f f(\phi_0(z_0))+\mu_f(f)\\
			 &=\lambda_f f(\phi_f(z_0))+\mu_f(f)\\
			 &=T(f)(z_0).
\end{align*}
Now, since $T(f)(y)=\mu(f)$, we deduce that $f(\phi_0(y_0))=0$ which is a contradiction.
 
Hence for each $y\in Y$, $\lambda^{-1}S_y$ is a nonzero complex homomorphism on $C(X)$. This easily implies that the map $S\colon C(X)\to C(Y)$ defined by 
$$
S(f)(y)=\lambda^{-1} S_y(f)=\lambda^{-1}\left(T(f)(y)-\mu(f)\right) \qquad (f\in C(X),\; y\in Y)
$$
is a unital homomorphism and, consequently, it is continuous, as well. Thus the restriction of its adjoint map to the maximal ideal space of $C(Y)$ induces a continuous map $\phi\colon Y\to X$ satisfying 
$$
S(f)(y)=f(\phi(y)) \qquad (f\in C(X),\; y\in Y).
$$
Hence $T(f)(y)=\lambda f(\phi(y))+\mu(f)$ for all $f\in C(X)$ and $y\in Y$. 
\end{proof}

\begin{cs} \label{Step 15}
$\phi\colon Y\to X$ is a homeomorphism. 
\end{cs}

\begin{proof}
First we show that $\phi$ is injective. For this, let $y_1,y_2$ be in $Y$ and assume that $\phi(y_1)=\phi(y_2)$. Clearly, the function $\phi$ is not constant since otherwise for each $f\in C(X)$, $T(f)$ would be a constant function on $Y$ and, since $T$ is diameter preserving, this is impossible if $f$ is not constant. Hence we can find a point $y_3\in Y$ such that $\phi(y_3)\neq \phi(y_1)$. Put $x_k=\phi(y_k)$ for $k=1,2,3$. Choose $f\in\mathcal{F}'_{\{x_1,x_3\}}$ and then, using Step \ref{Step 14}, we deduce that 
$$
T(f)(y_1)-T(f)(y_3)=\lambda\left(f(x_1)-f(x_3)\right),
$$
which implies that 
$$
\lambda_f\left(f(\phi_f(y_1))-f(\phi_f(y_3))\right)=\lambda\left(f(x_1)-f(x_3))\right).
$$
Thus $\{\phi_f(y_1),\phi_f(y_3)\}=\{x_1,x_3\}$. A similar argument shows that $\{\phi_f(y_2),\phi_f(y_3)\}=\{x_2,x_3\}$. Since $x_1=x_2$, these equalities imply that $\phi_f(y_1)=\phi_f(y_2)$ and, consequently, $y_1=y_2$. 

Now we show that $\phi$ is surjective. Assume on the contrary that there exists a point $x_0\in X \backslash \phi(Y)$. Being $\phi(Y)$ compact, we can choose a function $f\in C(X)$ satisfying $f(x_0)=1$ and $f=0$ on $\phi(Y)$. Then, using Step \ref{Step 14}, we get $T(f)(y)=\mu(f)$ for all $y\in Y$, that is, $T(f)$ is a constant function, a contradiction. 

It follows immediately that $\phi$ is a homeomorphism from $Y$ onto $X$. 
\end{proof}

\begin{cs}\label{Step 16}
$T$ is bijective. 
\end{cs}

\begin{proof}
We first prove that $T$ is injective. Let $f\in C(X)$ and assume $T(f)=0$. By Step \ref{Step 1}, $\diam(f)=\diam(T(f))=0$ and thus $f$ is a constant function. Hence $f=\alpha 1_X$ for some $\alpha\in\mathbb{C}$. Since $0=T(f)=T(\alpha 1_X)=\alpha T(1_X)$ and $T(1_X)=(\lambda+\mu(1_X))1_Y$, we obtain $\alpha=0$ and thus $f=0$. 

On the other hand, given $g\in C(Y)$, the function 
$$
f=\overline{\lambda}g\circ\phi^{-1}-\frac{\overline{\lambda}\mu(g\circ\phi^{-1})}{\lambda+\mu(1_X)}1_X
$$
belongs to $C(X)$ and $T(f)=g$. Indeed, we have
\begin{align*}
T(f)&=T\left(\overline{\lambda}g\circ\phi^{-1}-\frac{\overline{\lambda}\mu(g\circ\phi^{-1})}{\lambda+\mu(1_X)}1_X\right)\\
&=\lambda\left(\overline{\lambda}g\circ\phi^{-1}\circ\phi-\frac{\overline{\lambda}\mu(g\circ\phi^{-1})}{\lambda+\mu(1_X)}1_X\circ\phi\right)+\mu\left(\overline{\lambda}g\circ\phi^{-1}-\frac{\overline{\lambda}\mu(g\circ\phi^{-1})}{\lambda+\mu(1_X)}1_X\right)\\
&=g-\frac{\mu(g\circ\phi^{-1})}{\lambda+\mu(1_X)}1_Y+\overline{\lambda}\mu(g\circ\phi^{-1})1_Y-\frac{\overline{\lambda}\mu(g\circ\phi^{-1})\mu(1_X)}{\lambda+\mu(1_X)}1_Y=g.
\end{align*}
Hence $T$ is surjective. This completes the proof of Theorem \ref{main}.
\end{proof}

\textbf{Acknowledgements.} Research partially supported by Junta de Andaluc\'{\i}a grant FQM194.

\end{document}